\documentclass[12pt,letterpaper,twoside]{amsart}
\usepackage{amsmath,amssymb}
\usepackage{amsthm}
\usepackage[abbrev]{amsrefs}
\usepackage{latexsym}

\numberwithin{equation}{section}
\newtheorem{thm}{Theorem}[section]
\newtheorem{prop}[thm]{Proposition}
\newtheorem{lem}[thm]{Lemma}
\newtheorem{cor}[thm]{Corollary}
\newtheorem{ex}[thm]{Example}
\newtheorem{rem}[thm]{Remark}
\theoremstyle{definition} 
\newtheorem{dfn}[thm]{Definition}
\newcommand{\set}[1]{\{\,{#1}\,\}}
\newcommand{\Image}{\mathop{\rm Im}}
\DeclareMathOperator{\pr}{pr}
\DeclareMathOperator{\id}{id}

\newcommand{\cL}{\mathcal{L}}
\newcommand{\cM}{\mathcal{M}}
\newcommand{\cP}{\mathcal{P}}

\newcommand{\cX}{\mathcal{X}}
\newcommand{\cY}{\mathcal{Y}}

\newcommand{\field}[1]{\mathbb{#1}}

\newcommand{\R}{\field{R}}

\newcommand{\Z}{\field{Z}}
\newcommand{\N}{\field{N}}

\newcommand{\ep}{\varepsilon}

\newcommand{\mmsp}{mm-space}
\newcommand{\supp}{\mathop{\rm supp}}
\newcommand{\diam}{\mathop{\rm diam}}
\newcommand{\ObsDiam}{{\rm ObsDiam}}

\DeclareMathOperator{\dP}{{\it d}_{{\rm P}}}
\DeclareMathOperator{\kf}{{\it d}_{{\rm KF}}}

\DeclareMathOperator{\dconc}{{\it d}_{{\rm conc}}}

\newcommand{\Lip}{\mathcal{L}{\it ip}}

\newcommand{\dx}{{d_X}}

\newcommand{\mux}{{m_X}}
\newcommand{\muy}{{m_Y}}
\newcommand{\leb}{{\mathcal L^1}}
\newcommand{\Leb}[1]{\mathcal L^1|_{#1}}
\DeclareMathOperator{\ICL}{ICL}
\DeclareMathOperator{\IC}{IC}
\DeclareMathOperator{\dis}{dis}
\DeclareMathOperator{\dev}{dev}
\title{Isoperimetric inequality on a metric measure space and Lipschitz order with an additive error}
\author{Hiroki Nakajima}
\date{\today}
\begin{document}
\maketitle
\begin{abstract}
M. Gromov introduced the Lipschitz order relation on the set of metric 
measure spaces and developed a rich theory. In particular, he claimed 
that an isoperimetric inequality on a non-discrete space is represented 
by using the Lipschitz order. We relax the definition of the Lipschitz 
order allowing an additive error to relate with an isoperimetric 
inequality on a discrete space. As an application, we obtain an 
isoperimetric inequality on the non-discrete $n$-dimensional $l^1$-cube 
by taking the limits of an isoperimetric inequality of the discrete $l^1
$-cubes.
\end{abstract}
\section{Introduction}
M. Gromov introduced the Lipschitz order relation on the set of metric measure spaces and developed a rich theory \cite{Gmv:green}. The aim of this paper is to relax the definition of Lipschitz order to adopt an additive error in order to expand the range of its applications. Especially, obtaining isoperimetric inequalities on various spaces is one of the most important applications.

One of the most famous isoperimetric inequalities is L\'evy's isoperimetric inequality (Theorem \ref{thm:levy_ineq}).  On a general metric measure space, we consider a L\'evy type isoperimetric inequality. Let $(X,d_X)$ be a complete separable metric space with a Borel probability measure $m_X$. We call such a triple $(X,d_X,m_X)$ an {\it mm-space} (which is an abbreviation of a metric measure space). If we say that $X$ is an mm-space, the metric and the measure are respetively indicated by $d_X$ and $m_X$.
\begin{dfn}[Isoperimetric comparison condition of L\'evy type]
We say that an mm-space $X$ satisfies the {\it isoperimetric comparison condition of L\'evy type} $\ICL_\ep(\nu)$ for a Borel probability measure $\nu$ on $\R$ and a real number $\ep\ge 0$ if we have $F_\nu(b)\le m_X(B_{b-a+\ep}(A))$ for any $a,b\in\supp\nu$ with $a\le b$ and for any Borel subset $A\subset X$ with $m_X(A)>0$ and $F_\nu(a)\le m_X(A)$, where $F_\nu(t):=\nu((-\infty,t])$ is the cumulative distribution function of $\nu$. We write $\ICL(\nu)$ as $\ICL_0(\nu)$ for simplicity.
\end{dfn}
The {\it $1$-measurement} of an mm-space $X$ is defined as
\[
\cM (X;1):=\set{\varphi_*\mux\mid \varphi :X\to \R \text{ $1$-Lipschitz function} },
\]
where $\varphi_*\mux$ is the push-forward measure of $\mux$ by $\varphi$ and a {\it $1$-Lipschitz function} is a Lipschitz continuous function with
Lipschitz constant less than or equal to one.
We denote by $\cP(\R)$ the set of all Borel probability measures on $\R$ and we see $\cM(X;1)\subset \cP(\R)$. In the case where $\nu\in \cM(X;1)$, the $\ICL(\nu)$ condition for $X$ means to have a sharp isoperimetric inequality on $X$.
The L\'evy's isoperimetric inequality is paraphrased by that $S^n(1)$ satisfies $\ICL(\xi_*m_{S^n(1)})$, where $\xi:S^n(1)\to\R$ is the distance function from one point.
$\cP(\R)$ has an order relation called the iso(perimetrically)-Lipschitz order.
\begin{dfn}[iso-Lipschitz order]
Let $\mu,\nu\in\cP(\R)$. We say that {\it $\mu$ iso-dominates $\nu$} and denote $\mu\succ' \nu$ if there exists a monotone non-decreasing 1-Lipschiz function $f:\supp\mu\to\supp\nu$ such that $f_*\mu=\nu$, where $\supp\mu$ is the support of $\mu$.
\end{dfn}
Gromov defined an iso-dominant using the iso-Lipschitz order and claimed that an iso-dominant recollects the isoperimetric inequality \cite{Gmv:isop}.
\begin{dfn}[iso-dominant \cite{Gmv:isop}]
We call a Borel probability measure an {\it iso-dominant} of an mm-space $X$ if it is an upper bound of $\cM(X;1)$ with respect to the iso-Lipschitz order $\succ'$.
\end{dfn}
We have the following relation between an iso-dominant and ICL.
\begin{thm}[\cite{NkjShioya:isop}]\label{thm:ns}
Let $X$ be an mm-space and $\nu$ a Borel probability measure on $\R$. Assume that the cumulative distribution function $F_\nu$ of $\nu$ is continuous. Then, $X$ satisfies $\ICL(\nu)$ if and only if $\nu$ is an iso-dominant of $X$.
\end{thm}
Gromov claimed a variant of Theorem \ref{thm:ns} without proof (see \cite{Gmv:isop} \S 9).
We focus on the continuity of $F_\nu$ in Theorem \ref{thm:ns}.
Without the continuity of $F_\nu$, we find the following counter example of Theorem \ref{thm:ns}.
We put $[k]:=\set{0,\dots,k-1}$ and consider the $n$-dimensional discrete cube $[k]^n$ equipped with the $l^1$-distance and the uniform measure, say $m_{[k]^n}$.
Then, $[k]^n$ satisfies $\ICL((d_0)_*m_{[k]^n})$, where $d_0$ is the distance function from the origin \cite{Bol:comp}. 
Since the cumulative distribution function of $(d_0)_*m_{[k]^n}$ is not continuous, we are not able to apply Theorem \ref{thm:ns} with $[k]^n$ as an mm-space $X$.
Moreover, $(d_0)_*m_{[k]^n}$ is not an iso-dominant of $[k]^n$.
However, we regard $(d_0)_*m_{[k]^n}$ as an iso-dominant of $[k]^n$ if we allow an error.
This is one of our motivations of introducing the iso-Lipschitz order with an error.

Now, we define the iso-Lipschitz order with an additive error using transport plan (Definition \ref{dfn:trans}) and the following iso-deviation.
\begin{dfn}[iso-deviation]
We define the {\it iso-deviation} $\dev_\succ$ of a subset $S\subset \R^2$ by
\[
\dev_\succ S:=\sup\set{y-y'-\max\set{x-x',0}\mid (x,y),(x',y')\in S}.
\]
\end{dfn}
The iso-deviation evaluates the deviation from the monotone non-decreasing and 1-Lipschitz property.
\begin{dfn}[iso-Lipschitz order $\succ'_{(s,t)}$ with error $(s,t)$]
Let $\mu$ and $\nu$ be two Borel probability measures on $\R$ and $s,t\ge 0$ two real numbers. We say that {\it $\mu$ iso-dominates $\nu$ with error $(s,t)$} and denote $\mu\succ'_{(s,t)}\nu$ if there exists a transport plan $\pi\in\Pi(\mu,\nu)$ and a Borel subset $S\subset \R^2$ such that $\dev_\succ S\le s$ and $1-\pi(S)\le t$.
\end{dfn}
The iso-Lipschitz order $\succ'_{(s,t)}$ with error $(s,t)$ satisfies some beneficial properties such as Theorems \ref{thm:order_dev:zero_equiv}, \ref{thm:order_dev:nonDegenerate}, and \ref{thm:order_dev:transitive} in Section \ref{section:isoLipschitzOrder}.
Now, we define the iso-dominant with an error by using the iso-Lipschitz order with an error.
\begin{dfn}[$\ep$-iso-dominant]
Let $\ep\ge 0$ be a real number.
We call a Borel probability measure $\nu$ on $\R$ an {\it $\ep$-iso-dominant} of an mm-space $X$ if we have $\nu\succ'_{(\ep,0)}\mu$ for all $\mu\in\cM(X;1)$.
\end{dfn}
We have the following Theorem \ref{thm:isoICL}, which explains the relation between $\ep$-iso-dominant and $\ICL_\ep(\nu)$.
\begin{thm}\label{thm:isoICL}
Let $X$ be an mm-space and $\nu$ a Borel probability measure on $\R$, and let $\ep\ge 0$. We define 
\[
\Delta(\supp\nu):=\sup\set{\delta_\nu^-(a)\mid a\in\supp\nu\setminus\{\inf\supp\nu\}},
\]
where $\delta_\nu^-(a):=\inf\set{t>0\mid a-t\in\supp\nu}$. Then we have the following \eqref{thm:isoICL:ICLtoDominant} and \eqref{thm:isoICL:DominantToICL}.
\begin{enumerate}
\item If $\inf\supp\nu>-\infty$, we assume $\nu(\{\inf\supp\nu\})\le\mux(\{x\})$ for $x\in\supp\mux$. Then, $\nu$ is an $(\ep+\Delta(\supp\nu))$-iso-dominant of $X$ if $X$ satisfies $\ICL_\ep(\nu)$.\label{thm:isoICL:ICLtoDominant}
\item We assume that $\supp\nu$ is connected or $\nu(\{x\})>0$ for all $x\in\supp\nu$．Then，$X$ satisfies $\ICL_{2\ep}(\nu)$ if $\nu$ is an $\ep$-iso-dominant of $X$.
\label{thm:isoICL:DominantToICL}
\end{enumerate}
\end{thm}

The condition that $\nu$ is an $\ep$-iso-dominant of $X$ is stable under convergence with respect to the Prohorov distance $\dP$ and the observable distance $\dconc$. This property enables us to obtain the isoperimetric inequality of a continuous space by using  a discretization.
The following Theorem \ref{thm:stabilityIso} is one of the main theorem of this paper and represents the stability of $\ep$-iso-dominant.
\begin{thm}\label{thm:stabilityIso}
Let $X$ and $X_n, n=1,2,\dots$, be mm-spaces, $\nu$ and $\nu_n, n=1,2,\dots$, Borel probability measures, and $\ep_n, n=1,2,\dots$, non-negative real numbers. We assume that $X_n$ $\dconc$-converges to $X$ and $\nu_n$ weakly converges to $\nu$, and $\ep_n$ converges to a real number $\ep$ as $n\to\infty$ and that $\nu_n$ is an $\ep_n$-iso-dominant of $X_n$ for any positive integer $n$. Then, $\nu$ is an $\ep$-iso-dominant of $X$.
\end{thm}

We obtain a sharp isoperimetric inequality of the $n$-dimensional $l^1$-hyper cube $[0,1]^n$ as one of the applications of the Lipschitz order with an error by using Theorem 8 in \cite{Bol:comp}.
The $n$-dimensional $l^1$-hyper cube $[0,1]^n$ is the $n$-dimensional cube $[0,1]^n$ equipped with the $l^1$-distance $d_{l^1}$ and the uniform measure.
The following Theorem \ref{thm:lOneIso} is a sharp isoperimetric inequality on it.
\begin{thm}\label{thm:lOneIso}
$(d_0)_*m_{[0,1]^n}$ is the maximum of $\cM([0,1]^n;1)$, where $d_0$ is the distance function from the origin.
\end{thm}

By Theorems \ref{thm:lOneIso} and \ref{thm:ns}, the $l^1$-hyper cube $[0,1]^n$ satisfies 
$\ICL((d_0)_*$\\
$m_{[0,1]^n})$. Namely, we have the following Corollary \ref{cor:lOneIso}.
\begin{cor}\label{cor:lOneIso}
For any closed subset $\Omega\subset [0,1]^n$ with $\cL^n(\Omega)>0$, we take a metric ball $B_\Omega\subset [0,1]^n$ centered at the origin with $\cL^n(B_\Omega)=\cL^n(\Omega)$. Then we have 
\[
\cL^n|_{[0,1]^n}(U_r(\Omega))\ge\cL^n|_{[0,1]^n}(U_r(B_\Omega))
\]
for any $r>0$, where $U_r(A):=\set{x\in [0,1]^n\mid d_{l^1}(x,A)< r}$ is the open $r$-neighborhood of a subset $A\subset [0,1]^n$ with respect to the $l^1$-distance $d_{l^1}$.
\end{cor}
Similarly, we obtain the sharp isoperimetric inequality of the $l^1$-torus $T^n$ by using Corollary 6 in \cite{Bol:isop_torus}. The $l^1$-torus $T^n$ is the $n$-times $l^1$-product of one-dimensional sphere $S^1$ equipped with the uniform measure.
\begin{thm}\label{thm:lOneTorusIso}
$\xi_*m_{T^n}$ is the maximum of $\cM(T^n;1)$, where $\xi$ is the distance function from one point.
\end{thm} 
\begin{cor}
For any closed subset $\Omega\subset T^n$ with $m_{T^n}(\Omega)>0$, we take a metric ball $B_\Omega$ of $T^n$ with $m_{T^n}(B_\Omega)=m_{T^n}(\Omega)$. Then we have 
\[
m_{T^n}(U_r(\Omega))\ge m_{T^n}(U_r(B_\Omega))
\]
for any $r>0$, where $U_r(A)$ is the open $r$-neighborhood of a subset $A\subset T^n$ with respect to the $l^1$-distance.
\end{cor}
If the $1$-measurement $\cM(X;1)$ of an mm-space $X$ has the maximum element $\nu$, we obtain the precise value of the observable diameter $\ObsDiam(X;-\kappa)$ of  $X$ (Definition \ref{dfn:obsDiam}) because we have 
\[
\ObsDiam(X;-\kappa)=\diam(\nu;1-\kappa) \text{\quad for any $\kappa\in (0,1]$}.
\]
Thus, we obtain the value of $\ObsDiam([0,1]^n;-\kappa)$ and $\ObsDiam(T^n;-\kappa)$ for any $\kappa\in (0,1]$.
As former results, the $n$-dimensional unit sphere is known to be an mm-space whose $1$-measurement has the maximum element (see \S 9 in \cite{Gmv:isop}). The $n$-dimensional Gaussian space is also such an mm-space because of an isoperimetric inequality \cite{Bor:gauss,Sud:gauss}.

As another application of Theorem \ref{thm:stabilityIso}, we obtain the following, which is a variant of normal law \`a la L\'evy (see Theorem 2.2 in \cite{Shioya:mmg}) by using Theorem 13 in \cite{Bol:comp}. 
\begin{thm}[Normal law \`a la L\'evy on product graphs]\label{thm:NormalLevyProductGraphs}
Let $G_1,G_2,$\\
$\dots,G_n,\dots$ be connected graphs with same order $k\ge 2$. Put 
\[
\varepsilon_n:=\sqrt{\frac{12}{(k^2-1)n}}.
\]
Let $X_n:=(\prod_{i=1}^n G_i,d_{X_n},m_{X_n})$ be the cartesian product graph equipped with the path metric $d_{X_n}$ and the uniform measure $m_{X_n}$. Put $Y_n:=(\prod_{i=1}^n G_i,\ep_n\cdot d_{X_n},m_{X_n})$. Let $\{f_{n_i}\}$ be a subsequence of a sequence of 1-Lipschitz functions $f_n:Y_n\to \R, n=1,2,\dots$. If $(f_{n_i})_*m_{Y_{n_i}}$ converges weakly to a Borel probability measure $\sigma$, then we have $\gamma^1\succ'\sigma$, where $\gamma^1$ is the $1$-dimensionnal Gaussian measure.
\end{thm}
In the case that $k=2$, we see that $X_n$ is the $n$-dimensionnal Hamming cube. If we replace $X_n$ by $n$-dimensional (non-discrete) $l^1$-cube or $n$-dimansional (non-discrete) $l^1$-torus, we obtain normal law \`a la L\'evy respectively.

\section{Preliminaries}\label{preliminaries}

In this section, we present some basics of mm-space.
We refer to \cite{Gmv:green,Shioya:mmg} for more details about this section.
\subsection{Some basics of \mmsp}\label{basics}
\begin{dfn}[\mmsp]
Let $(X,\dx)$ be a complete separable metric space and $m_X$ a Borel probability measure on $X$.
We call such a triple $(X,\dx,\mux)$ an {\it\mmsp}.
We sometimes say that $X$ is an \mmsp, for which the metric and measure of $X$ are respectively indicated by $\dx$ and $\mux$. We put $tX:=(X, td_X,m_X)$ for $t>0$.
\end{dfn}

We denote the Borel $\sigma$-algebra over $X$ by $\mathcal B_X$.
For any point $x\in X$, any two subsets $A,B\subset X$ and any real number $r>0$, we define
\begin{align*}
d_X(x,A)&:=\inf_{y\in A} d_X(x,y),\\
d_X(A,B)&:=\inf_{x\in A,\,y\in B}d_X(x,y),\\
U_r(A)&:=\set{y\in X\mid d_X(y,A)<r},\\
B_r(A)&:=\set{y\in X\mid d_X(y,A)\leq r}.
\end{align*}

Let $p:X\to Y$ be a measurable map from a measure space $(X,\mux)$ to a topological space $Y$. {\it The push-forward of $\mux$ by the map $p$} is defined as $p_*\mux(A):=\mux(p^{-1}(A))$ for any $A\in \mathcal B_Y$.
\begin{dfn}[mm-isomorphism]
Two \mmsp s $X$ and $Y$ are said to be {\it mm-isomorphic} to each other if there exists an isometry $f:\supp\mux\to\supp\muy$ such that $f_*\mux=\muy$, where $\supp\mux$ is the support of $\mux$.
Such an isometry $f$ is called an {\it mm-isomorphism}.
The mm-isomorphism relation is an equivalence relation on the set of mm-spaces.
Denote by $\mathcal X$ the set of mm-isomorphism classes of \mmsp s.
\end{dfn}
\begin{dfn}[Lipschitz order]\label{def:Lip_ord}
Let $X$ and $Y$ be two \mmsp s.
We say that $X$ {\it dominates} $Y$ and write $Y\prec X$ if there exists a 1-Lipschitz map $f:X\to Y$ satisfying
\[
f_*\mux=\muy.
\]
We call the relation $\prec$ on $\mathcal X$ the {\it Lipschitz order}.
\end{dfn}
\begin{prop}[Proposition 2.11 in \cite{Shioya:mmg}]\label{prop:lipPartialOrder}
The Lipschitz order $\prec$ is a partial order relation on $\mathcal X$.
\end{prop}

\begin{dfn}[Transport plan]\label{dfn:trans}
Let $\mu$ and $\nu$ be two Borel probability measures on $\R$. We say that a Borel probability measure on $\R^2$ is a transport plan between $\mu$ and $\nu$ if we have $(\pr_1)_*\pi=\mu$ and $(\pr_2)_*\pi=\nu$, where $\pr_1$ and $\pr_2$ is the first and second projection respectively. We denote by $\Pi(\mu,\nu)$ the set of transport plans between $\mu$ and $\nu$.
\end{dfn}

\subsection{Observable diameter and partial diameter}
Observable diameter is one of the most important invariants. We remark that this is defined by the 1-measurement.
\begin{dfn}[Partial diameter]
Let $X$ be an \mmsp.
For any real number $\alpha\in[0,1]$, we define the {\it partial diameter $\diam(X;\alpha)=\diam(\mux;\alpha)$ of $X$} as
\[
\diam(X;\alpha):=\inf\set{\diam A\mid \mux(A)\geq \alpha, \, A\in\mathcal B_X},
\]
where the {\it diameter of $A$} is defined by $\diam A:=\sup_{x,y\in A}d_X(x,y)$ for $A\neq\emptyset$ and $\diam \emptyset:=0$.
\end{dfn}
\begin{dfn}[Observable diameter]\label{dfn:obsDiam}
Let $X$ be an \mmsp.
For any real number $\kappa\in[0,1]$, we define the {\it $\kappa$-observable diameter \\
$\ObsDiam(X;-\kappa)$ of $X$} as
\[
\ObsDiam(X;-\kappa):=\sup_{\mu\in\mathcal M(X;1)}\diam(\mu;1-\kappa).
\]
\end{dfn}
\begin{prop}[Proposition 2.18 in \cite{Shioya:mmg}]\label{prop:Lip_invariant}
Let $X$ and $Y$ be two \mmsp s and $\kappa\in[0,1]$ a real number.
If $Y\prec X$, then we obtain
\begin{align*}
\diam(Y;1-\kappa)&\leq\diam(X;1-\kappa),\\
\ObsDiam(Y;-\kappa)&\leq\ObsDiam(X;-\kappa).
\end{align*}
\end{prop}
\subsection{L{\'e}vy's isoperimetric inequality}
Let $S^n(r)$ be the $n$-dimensional sphere of radius $r>0$ centered at the origin in the $(n+1)$-dimensional Euclidean space $\R^{n+1}$. We assume the distance $d_{S^n(r)}(x,y)$ between two points $x$ and $y$ in $S^n(r)$ to be the geodesic distance and the measure $m_{S^n(r)}$ on $S^n(r)$ to be the Riemannian volume measure on $S^n(r)$ normalized as $m_{S^n(r)}(S^n(r))=1$.
Then, $(S^n(r),\,d_{S^n(r)},\,m_{S^n(r)})$ is an \mmsp.
\begin{thm}[L{\'e}vy's isoperimetric inequality \cite{Levy:iso,Milman:iso}]\label{thm:levy_ineq}
For any closed subset $\Omega\subset S^n(1)$, we take a metric ball $B_\Omega$ of $S^n(1)$ with $m_{S^n(1)}(B_\Omega)=m_{S^n(1)}(\Omega)$.
Then we have
\[
m_{S^n(1)}(U_r(\Omega))\geq m_{S^n(1)} (U_r(B_\Omega))
\]
for any $r>0$.
\end{thm}
\subsection{Box distance}
In this subsection, we briefly describe the box distance.
\begin{dfn}[Parameter]
Let $I:=[0,1)$ and let $\mathcal L^1$ be the one-dimensional Lebesgue measure on $I$.
Let $X$ be a topological space with a Borel probability measure $\mux$.
A map $\varphi:I\to X$ is called a {\it parameter of $X$} if $\varphi$ is a Borel measurable map such that
\[
\varphi_*\mathcal L^1=\mux.
\]
\end{dfn}
\begin{dfn}[Pseudo-metric]
A {\it pseudo-metric $\rho$ on a set $S$} is defined to be a function $\rho:S\times S\to[0,\infty)$ satisfying that, for any $x,y,z\in S$,
\begin{enumerate}
\item $\rho(x,x)=0,$
\item $\rho(y,x)=\rho(x,y),$
\item $\rho(x,z)\leq\rho(x,y)+\rho(y,z)$.
\end{enumerate}
\end{dfn}
\begin{dfn}[Box distance]\label{dfn:box}
For two pseudo-metrics $\rho_1$ and $\rho_2$ on $I:=[0,1)$, we define $\square(\rho_1,\rho_2)$ to be the infimum of $\varepsilon\geq 0$ satisfying that there exists a Borel subset $I_0\subset I$ such that
\begin{enumerate}
\item $|\rho_1(s,t)-\rho_2(s,t)|\leq\varepsilon$ for any $s,t\in I_0$,
\item $\mathcal L^1(I_0)\geq 1-\varepsilon$.
\end{enumerate}
We define the {\it box distance $\square(X,Y)$ between two \mmsp s $X$ and $Y$} to be the infimum of $\square(\varphi^*d_X,\psi^*d_Y)$, where $\varphi:I\to X$ and $\psi:I\to Y$ run over all parameters of $X$ and $Y$, respectively, and where $\varphi^*\dx(s,t):=\dx(\varphi(s),\varphi(t))$ for $s,t\in I$.
\end{dfn}
\begin{thm}[Theorem 4.10 in \cite{Shioya:mmg}]
The box distance $\square$ is a metric on the set $\mathcal X$ of mm-isomorphism classes of \mmsp s.
\end{thm}

\begin{dfn}[Obsevable distance]
For two Borel measurable maps $f,g:I:=[0,1)\to\R$, we define the {\it Ky Fan metric} $d_{\rm KF}$ by
\[
d_{\rm KF}(f,g):=\inf\set{\ep\ge 0\mid \cL^1(\set{t\in I\mid |f(t)-g(t)|>\ep})\le\ep}.
\]
For a parameter $\varphi$ of an mm-space $X$, we define 
\[
\varphi^*\Lip_1(X):=\set{f\circ \varphi\mid \text{$f:X\to \R$ is 1-Lipschitz}}.
\]
The Hausdorff distance $d_{\rm H}^{\rm KF}$ is defined with respect to $d_{\rm KF}$.
We define the {\it observable distance} $\dconc$ between two mm-spaces $X$ and $Y$ by
\[
\dconc(X,Y):=\inf_{\varphi,\psi} d_{\rm H}^{\rm KF}(\varphi^*\Lip_1(X),\psi^*\Lip_1(Y))
\]
where $\varphi:I\to X$ and $\psi:I\to Y$ are two parameters of $X$ and $Y$ respectively.
\end{dfn}
\begin{thm}[Theorem 5.13 in \cite{Shioya:mmg}]
$\dconc$ is a metric on $\cX$.
\end{thm}
\begin{prop}[Proposition 5.5 in \cite{Shioya:mmg}]\label{prop:concBox}
For two mm-spaces $X$ and $Y$, we have
$\dconc(X,Y)\le \square (X,Y)$.
\end{prop}
\section{The iso-Lipschitz order with an error}\label{section:isoLipschitzOrder}

In this section, we present some properties of the iso-Lipschitz order with an error.

\begin{dfn}[iso-mm-isomorphic]
Two Borel probability measures $\mu$ and $\nu$ on $\R$ are said to be {\it iso-mm-isomorphic} to each other if there exists a real number $c$ such that $(\id_\R+c)_*\mu=\nu$, where $\id_\R$ is the identity function on $\R$. The iso-mm-isomorphic relation is an equivalence relation on the set of Borel probability measures on $\R$.
\end{dfn}

\begin{prop}
The iso-Lipschitz order is a partial order on the set of iso-mm-isomorphism class of Borel probability measures on $\R$.
\end{prop}

\begin{prop}\label{prop:dev_closure}
For a subset $S\subset \R^2$, we have
\[
\dev_\succ S=\dev_\succ\overline{S}.
\]
\end{prop}

\begin{lem}\label{lem:dev_minus}
Let $S\subset \R^2$. For any two points $(x,y),(x',y')\in S$, we have
\[
|y-y'|-|x-x'|\le \dev_\succ S.
\]
\end{lem}
\begin{proof}
Take any $(x,y),(x',y')\in S$. By symmetry, we may assume that $y\ge y'$. Then we have
\[
|y-y'|-|x-x'|\le y-y'-\max\set{x-x',0}\le\dev_\succ S.
\]
This completes proof.
\end{proof}
\begin{thm}\label{thm:order_dev:zero_equiv}
Let $\mu$ and $\nu$ be two Borel probability measures on $\R$. Then we have $\mu\succ'\nu$ if and only if $\mu\succ'_{(0,0)}\nu$.
\end{thm}

\begin{proof}
Assume that $\mu\succ'\nu$. Then, there exists a monotone non-decreasing 1-Lipschitz function $f:\supp\mu\to\supp\nu$ such that $f_*\mu=\nu$.
We put $\pi:=(\id_\R,f)_*\mu\in\Pi(\mu,\nu)$. Let us prove $\dev_\succ\supp\pi=0$.
It suffices to prove $\dev_\succ ((\id_\R,f)(\supp\mu))=0$ because of Proposition \ref{prop:dev_closure} and $\supp\pi=\overline{(\id_\R,f)(\supp\mu)}$. 
Take any two points $(x_1,y_1),(x_2,y_2)\in\supp\pi=(\id_\R,f)(\supp\mu)$.
Then, we have $x_1,x_2\in \supp\mu$ and $y_1=f(x_1), y_2=f(x_2)$. In the case that $x_1\ge x_2$, we have 
\begin{align*}
y_1-y_2-\max\set{x_1-x_2,0}&=f(x_1)-f(x_2)-|x_1-x_2|\\
&\le|f(x_1)-f(x_2)|-|x_1-x_2|\le 0
\end{align*}
because $f$ is 1-Lipschitz. In the case that $x_1\le x_2$, we have $f(x_1)\le f(x_2)$ since $f$ is monotone non-decreasing.
Then we have
\begin{align*}
y_1-y_2-\max\set{x_1-x_2,0}&=y_1-y_2
=f(x_1)-f(x_2)\le 0.
\end{align*}
Therefore we obtain $\dev_\succ\supp\pi= 0$.
It follows that $\mu\succ_{(0,0)}\nu$.

Conversely, assume that $\mu\succ_{(0,0)}\nu$.
Then there exists $\pi\in\Pi(\mu,\nu)$ such that $\dev_\succ\supp\pi=0$.
Now, for any $x\in\supp\mu$, there exists a unique point $y\in \supp \nu$ such that $(x,y)\in\supp\pi$.
Let us prove the existence of $y$. Take any $x\in\supp\mu$. 
Since we have
\[
\supp\mu=\supp (\mathrm{pr}_1)_*\mu=\overline{\mathrm{pr}_1(\supp\pi)},
\]
there exists $\{(x_n,y_n) \}_{n\in\N}\subset\supp\pi$ such that $x_n$ converges to $x$. 
By Proposition \ref{lem:dev_minus}, we have
\[
|y_m-y_n|-|x_m,x_n|\le\dev_\succ\supp\pi=0
\]
for any positive integers $m$ and $n$. This means that $\{y_n \}$ is a Cauchy sequence. Therefore, $\{y_n \}$ converges to some $y\in\R$. Since $\supp\pi$ is closed, we have $(x,y)\in\supp\pi$. 
In addition, we have
\[
y\in\mathrm{pr}_2(\supp\pi)\subset \supp (\mathrm{pr}_2)_*\pi=\supp \nu.
\]
The uniqueness of $y\in\supp\nu$ follows from $\dev_\succ\supp\pi=0$ and Proposition \ref{lem:dev_minus}. Now, we define a function $f:\supp\mu\to\supp\nu$ by $f(x):=y$ for $x\in\supp\mu$, where $y\in\supp\nu$ satisfies $(x,y)\in\supp\pi$. By $\dev_\succ\supp\pi=0$ and Proposition \ref{lem:dev_minus}, $f$ is a $1$-Lipschtiz function. Let us prove that $f$ is monotone non-decreasing. Take any $x,x'\in\supp\mu$ with $x\le x'$. Then we have
\[
f(x)-f(x')=f(x)-f(x')-\max\set{x-x',0}\le\dev_\succ\supp\pi=0.
\]
The rest of the proof is to show $f_*\mu=\nu$.
Now, we have $\supp\pi=\set{(x,f(x))\mid x\in\supp\mu}$ by the definition of $f$. Therefore, we have
\[
(A\times B)\cap \supp\pi=\{(A\cap f^{-1}(B))\times Y \}\cap \supp\pi
\]
for any Borel sets $A\subset X$ and $B\subset Y$.
Since
\begin{align*}
\pi(A\times B)&=\pi((A\cap f^{-1}(B))\times Y)\\
&=\mu(A
\cap f^{-1}(B))\\
&=(\mathrm{id}_\R,f)_*\mu(A\times B),
\end{align*}
we have $\pi=(\mathrm{id}_\R,f)_*\mu$, which implies $\nu=(\mathrm{pr}_2)_*\pi=f_*\mu$. This completes the proof.
\end{proof}

\begin{prop}\label{prop:order_dev:deviationConti}
Let $d_{l^1}$ be the $l^1$-distance $d_{l^1}((x,y),(x',y')):=|x-x'|+|y-y'|$ on $\R^2$ and $d_H$ the Hausdorff distance with respect to $d_{l^1}$. For any two closed subsets $S,S'\subset \R^2$, we have
\[
|\dev_\succ S-\dev_\succ S'|\le 2d_H(S,S').
\]
\end{prop}
\begin{proof}
Take any real number $\ep>0$ with $\varepsilon>d_H(S,S')$. We have $S'\subset U_\varepsilon (S)$. Let us prove $\dev_\succ U_\varepsilon(S)\le \dev_\succ S+2\varepsilon$. Take a point $(x_i,y_i)\in U_\varepsilon (S)$ for $i=1,2$.
Then there exists $(x'_i,y'_i)\in S$ such that $d_{l^1}((x_i,y_i),(x'_i,y'_i))<\varepsilon$.
Now, we have
\begin{align*}
&y_1-y_2-\max\set{x_1-x_2,0}\\
&= y'_1-y'_2+(y_1-y'_1)+(y_2-y'_2)\\
&\qquad -\max\set{x'_1-x'_2+(x_1-x'_1)+(x'_1-x'_2),0}\\
&\le y'_1-y'_2+|y_1-y'_1|+|y_2-y'_2|\\
&\qquad-\max\set{x'_1-x'_2-|x_1-x'_1|-|x'_1-x'_2|,0}\\
&\le y'_1-y'_2+|y_1-y'_1|+|y_2-y'_2|\\
&\qquad-(\max\set{x'_1-x'_2,0}-|x_1-x'_1|-|x'_1-x'_2|)\\
&\le y'_1-y'_2-\max\set{x'_1-x'_2,0}+2\varepsilon\\
&\le \dev_\succ S+2\varepsilon.
\end{align*}
Therefore we obtain
\[
\dev_\succ S'\le \dev_\succ U_\varepsilon(S)\le \dev_\succ S+2\varepsilon.
\]
This implies $\dev_\succ S'-\dev_\succ S\le 2d_H(S,S')$. By exchanging $S$ for $S'$, we also obtain $\dev_\succ S-\dev_\succ S'\le 2d_H(S,S')$.
\end{proof}

\begin{thm}\label{thm:order_dev:nonDegenerate}
Let $\mu$ and $\nu$ be two Borel probability measures on $\R$ and $s,t\ge 0$. If $\mu\succ'_{(s+\ep,t+\ep)}\nu$ for any $\ep>0$, then we have $\mu\succ'_{(s,t)}\nu$. 
\end{thm}

\begin{proof}
Suppose that $\mu\succ'_{(s+\frac 1n,t+\frac 1n)}\nu$ for any positive integer $n$.
For any positive integer $n$, there exist $\pi_n\in\Pi(\mu,\nu)$ and a closed subset $S_n\subset \R^2$ such that $\dev_\succ S_n\le s+\frac 1n$ and $\pi_n(S_n)\ge 1-t-\frac 1n$.
Due to the weakly compactness of $\Pi(\mu,\nu)$, we may assume that $\pi_n$ converges weakly to some Borel probability measure $\pi$ by taking a subsequence.
By Prohorov's theorem, for any positive number $m$, there exists a compact subset $K_m\subset \R^2$ such that $\sup_{n\in\N}\pi_n(K_m^c)\le\frac 1m$ and $\pi(K_m^c)\le \frac 1m$. We may assume that the sequence of $\{K_m\}$ is monotone non-decreasing with respect to the inclusion relation.
Let $d_H$ be the Hausdorff distance of $(\R^2, d_{l^1})$ and $d_H^m$ the Hausdorff distance of $(K_m,d_{l^1})$.
Since $K_m$ is compact, $(\mathcal F (K_m),d_H^m)$ is also compact.
By taking a subsequence $\{n_i^1\}_{i\in\N}\subset \N$, we have $d_H^1(S_{n_i^1}\cap K_1,S_\infty^1)\to 0$ as $i\to\infty$, where $\N$ is the set of positive integers.
Furthermore, we take some subsequence $\{n_i^2\}_{i\in\N}\subset \{n_i^1\}_{i\in\N}$ and we have $d_H^2(S_{n_i^2}\cap K_2,S_\infty^2)\to 0$. By repeating this procedure, we take a subsequence $\{n_i^m\}_{i\in\N}\subset \{n_i^{m-1}\}_{i\in\N}$ and we have $d_H^m(S_{n_i^m}\cap K_m,S_\infty^m)\to 0$.
Since the convergence on $(\mathcal F(K_m),d_H^m)$ implies the convergence on $(\mathcal F(\R),d_H)$, we obtain
\begin{equation}
d_H(S_{n_i}\cap K_m,S_\infty^m)\to 0 \label{eqn:order_dev:nonDegenerate:hausdorff}
\end{equation}
for any positive integer $m$.
Since $\{K_m\}$ is monotone non-decreasing with respect to inclusion relation, $\{S_\infty^m\}$ is also monotone non-decreasing.
By Proposition \ref{prop:order_dev:deviationConti} and\eqref{eqn:order_dev:nonDegenerate:hausdorff}, we have
\begin{equation}\label{eqn:order_dev:nonDegenerate:SmDis}
\begin{aligned}
\dev_\succ S_\infty^m&\le \liminf_{i\to\infty}\dev_\succ (S_{n_i}\cap K_m)\\
&\le \liminf_{i\to\infty}\dev_\succ (S_{n_i})\\
&\le \liminf_{i\to\infty}(s+\frac 1{n_i})= s
\end{aligned}
\end{equation}
Since $\{\pi_{n_i}\}$ converges weakly to $\pi$ and \eqref{eqn:order_dev:nonDegenerate:hausdorff}, we also have
\begin{equation}\label{eqn:order_dev:nonDegenerate:piSm}
\begin{aligned}
\pi(S_\infty^m)&\ge \limsup_{i\to\infty} \pi_{n_i}(S_{n_i}\cap K_m)\\
&= \limsup_{i\to\infty} (\pi_{n_i}(S_{n_i})-\pi_{n_i}(S_{n_i}\cap K_m^c))\\
&\ge \limsup_{i\to\infty} (\pi_{n_i}(S_{n_i})-\pi_{n_i}(K_m^c))\\
&\ge \limsup_{i\to\infty} (1-t-\frac 1{n_i}-\frac 1m)= 1-t-\frac 1m
\end{aligned}
\end{equation}
for any positive number $m$.
Now, we put $S:=\bigcup_{m=1}^\infty S_\infty^m $.
By \eqref{eqn:order_dev:nonDegenerate:SmDis}, we have
\[
\dev_\succ S= \sup_{m\in \N} \dev_\succ S_\infty^m\le s
\]
By \eqref{eqn:order_dev:nonDegenerate:piSm}, we have
\[
\pi(S)=\lim_{m\to\infty}\pi( S_\infty^m)
\ge \lim_{m\to\infty}(1-t-\frac 1m)
= 1-t.
\]
Therefore we obtain $\mu\succ'_{(s,t)} \nu$.
This completes the proof.
\end{proof}
\begin{dfn}[Subtransport plan]
Let $\mu$ and $\nu$ be two Borel probability measures on $\R$. We say that a Borel measure on $\R^2$ is a subtransport plan between $\mu$ and $\nu$ if we have $(\pr_1)_*\pi\le\mu$ and $(\pr_2)_*\pi\le\nu$.
\end{dfn}
\begin{prop}
Let $\mu$ and $\nu$ be two Borel probabilty measures on $\R$.
Then we have $\mu\succ'_{(s,t)}\nu$ if and only if there exists a subtransport plan $\pi$ between $\mu$ and $\nu$ such that $\dev_\succ \supp\pi\le s$ and $1-\pi(\R^2)\le t$.
\end{prop}
\begin{thm}\label{thm:order_dev:transitive}
Let $\mu_1$, $\mu_2$, and $\mu_3$ be three Borel probability measures on $\R$ and let $s_i,t_i\ge 0$ for $i=1,2$. 
If $\mu_1\succ'_{(s_1,t_1)}\mu_2$ and if $\mu_2\succ'_{(s_2,t_2)}\mu_3$, then we have $\mu_1\succ'_{(s_1+s_2,t_1+t_2)}\mu_3$.
\end{thm}
\begin{proof}
Suppose that $\mu_1\succ'_{(s_1,t_1)}\mu_2$ and $\mu_2\succ'_{(s_2,t_2)}\mu_3$.
There exists a subtransport plan $\pi_i$ between $\mu_i$ and $\mu_{i+1}$ such that $\dev_\succ\supp\pi_i\le s_i$ and $1-\pi_i(\supp\pi_i)\le t_i$ for $i=1,2$.
Put $\mu':=(\mathrm{pr}_2)_*\pi_1$ and $\mu'':=(\mathrm{pr}_1)_*\pi_2$.
By the disintegration theorem, there exist two families $\{(\pi_1)_x\}_{x\in\R}$ and $\{(\pi_2)_x\}_{x\in\R}$ of Borel measures on $\R$ such that
\begin{align*}
\pi_1(A\times B)&=\int_{B} (\pi_1)_x(A) d\mu'(x),\\
\pi_2(A\times B)&=\int_{A} (\pi_2)_x(B) d\mu''(x)
\end{align*}
for any Borel subsets $A$ and $B$ of $\R$. Now, we put
\begin{align*}
\pi_{123}(A\times B\times C)&:=\int_{B}(\pi_1)_x(A)\cdot (\pi_2)_x(C) d(\mu'\wedge\mu'')(x),\\
\pi_{13}:=(\mathrm{pr}_{13})_*\pi_{123}
\end{align*}
for any three Borel subsets $A$, $B$, and $C$ of $\R$, where $\mu'\wedge\mu'':=\mu'-(\mu'-\mu'')_+$ and we define the measure $(\mu'-\mu'')_+$ by
\[
(\mu'-\mu'')_+(B):=\sup\set{\mu'(B')-\mu''(B')\mid \text{$B'\subset B$ is a Borel set}}
\]
for any Borel set $B\subset \R$.
Then we have
\begin{equation}\label{eq:thm:order_dev:transitive:marginal}
(\mathrm{pr}_{12})_*\pi_{123}\le \pi_1,\quad
(\mathrm{pr}_{23})_*\pi_{123}\le \pi_2.
\end{equation}
In particular, $\pi_{13}$ is a subtransport plan between 
$\mu_1$ and $\mu_3$.
Moreover, we obtain $1-\pi_{13}(\supp\pi_{13})\le t_1+t_2$.
In fact, we have
\begin{align*}
\pi_{13}(\R^2)
&=\int_\R\left((\pi_1)_x(\R)\cdot (\pi_2)_x(\R)\right)d(\mu'\wedge \mu'')(x)\\
&=(\mu'\wedge \mu'')(\R)\\
&=\mu'(\R)-(\mu'-\mu'')_+(\R)\\
&\geq \mu'(\R)-(\mu_2-\mu'')_+(\R)\\
&=\mu'(\R)-(\mu_2(\R)-\mu''(\R))\\
&=\mu'(\R)+\mu''(\R)-1\\
&\geq (1-t_1)+(1-t_2)-1=1-t_1-t_2.
\end{align*}
The rest of the proof is to show $\dev_\succ \supp\pi_{13}\le s_1+s_2$. By Proposition \ref{prop:dev_closure} and $\supp\pi_{13}=\overline{\mathrm{pr}_{13}(\supp\pi_{123})}$, it suffices to prove 
\begin{equation}\label{eq:thm:order_dev:transitive:devEval}
\dev_\succ (\mathrm{pr}_{13}(\supp\pi_{123}))\le s_1+s_2.
\end{equation}
Take any $(x_i,z_i)\in\mathrm{pr}_{13}(\supp\pi_{123})$ for $i=1,2$.
There exists a point $y_i\in \R$ such that $(x_i,y_i,z_i)\in \supp\pi_{123}$.
By \eqref{eq:thm:order_dev:transitive:marginal}, we have
\[
\supp\pi_{123}\subset \mathrm{pr}_{12}^{-1}(\overline{\mathrm{pr}_{12}(\supp\pi_{123})})
\subset \mathrm{pr}_{12}^{-1}(\supp\pi_1)
\]
and 
\[
\supp\pi_{123}\subset \mathrm{pr}_{23}^{-1}(\overline{\mathrm{pr}_{23}(\supp\pi_{123})})
\subset \mathrm{pr}_{23}^{-1}(\supp\pi_2).
\]
This implies that $(x_i,y_i)\in\supp\pi_1$ and $(y_i,z_i)\in\supp\pi_2$.
Now, let us prove
\begin{equation}\label{eq:thm:order_dev:transitive:mainEq}
\max\set{y_1-y_2,0}-\max\set{x_1-x_2,0}\le s_1.
\end{equation}
In the case that $y_1<y_2$, we have
\[
\max\set{y_1-y_2,0}-\max\set{x_1-x_2,0}=-\max\set{x_1-x_2,0}\le 0.
\]
In the case that $y_1\ge y_2$, we have
\begin{align*}
\max\set{y_1-y_2,0}-\max\set{x_1-x_2,0}&=y_1-y_2-\max\set{x_1-x_2}\\
&\le \dev_\succ \supp \pi_1\le s_1.
\end{align*}
Since \eqref{eq:thm:order_dev:transitive:mainEq} and $\dev_\succ\supp \pi_2\le s_2$, we obtain
\begin{align*}
z_1-z_2-\max\set{x_1-x_2,0}&\le z_1- z_2-\max\set{y_1-y_2,0}+s_1\\
&\le s_1+s_2,
\end{align*}
which implies \eqref{eq:thm:order_dev:transitive:devEval}. This completes the proof.
\end{proof}

\section{the isoperimetric comparison condition with an error}

In this section, we prove Theorem \ref{thm:isoICL} to explain the relation between $\ep$-iso-dominant and $\ICL_\ep$. We also explain the relation between $\IC_\ep^+$ (Definition \ref{dfn:IC_plus}) and $\ICL_\ep$. $\IC_\ep^+$ is a discretization of $\IC$ in \cite{NkjShioya:isop}. At the end of this section, we give some examples of these conditions.

\begin{prop}\label{prop:isoScaling}
Let $\ep$ be a non-negative real number.
If a Borel probability measure $\nu$ on $\R$ is an $\varepsilon$-iso-dominant of an mm-space $X$, then $(t\cdot\mathrm{id}_\R)_*\nu$ is a $t\varepsilon$-iso-dominant of $tX$.
\end{prop}

\begin{rem}
By Theorem \ref{thm:order_dev:zero_equiv}, a Borel measure on $\R$ is a $0$-iso-dominant if and only if it is an iso-dominant．
\end{rem}

\begin{dfn}[$\varepsilon$-Discrete isoperimetric profile]
Let $X$ be an mm-space, and $\ep\ge 0$ a real number.
We define the {\it $\varepsilon$-discrete isoperimetric profile $I_X^\varepsilon$ of $X$} by
\[
I_X^\varepsilon(v):=\inf\set{\mux(B_\varepsilon (A))\mid \mux(A)=v} \text{ for $v\in \Image\mux$},
\]
where $\Image\mux:=\set{m_X(A) \mid \text{$A\subset X$ is a Borel set.} }$.
\end{dfn}

\begin{dfn}[Isoperimetric comparison condition with an  error]
\label{dfn:IC_plus}\ \\
We say that an mm-space $X$ satisfies the condition $\IC_\varepsilon^+(\nu)$  for a Borel probability measure $\nu$ on $\R$ and a real number $\ep\ge 0$ if we have
\[
I_X^{\delta^+(t)+\varepsilon}\circ V(t)\ge V(t+\delta^+(t))
\]
for any $t\in (\supp\nu\setminus\{\sup\supp\nu\})\cap V^{-1}(\Image\mux)$, where
\[
\delta^+(t):=\inf\set{s>0\mid t+s\in\supp\nu}.
\]
\end{dfn}

Now, we prepare some definitions for the proof of Theorem \ref{thm:isoICL}.

\begin{dfn}[Generalized inverse function]
For a monotone nondecreasing and right-continuous function $F :
\R\to [ 0, 1 ]$ with 
\[
\lim_{t\to -\infty} F(t) = 0,
\]
we define a generalized inverse function $\tilde F : [ 0, 1 ] \to \R$ by
\[
\tilde F(s) :=
\begin{cases}
\inf\set{ t \in \R \mid s \le F(t) } & \text{ if $s \in ( 0, 1 ]$},\\
c &\text{ if $s = 0$}
\end{cases}
\]
for $s \in [ 0, 1 ]$, where $c$ is a real constant.
\end{dfn}
Let $A$ be a subset of $\R$.
We put
\[
\delta^-(A;a):=\inf\set{t>0\mid a-t\in A}
\]
for a point $a\in A$, where we define
\[
\delta^-(A;a):=\infty
\]
if $\set{t>0\mid a-t\in A}=\emptyset$.
We define $\Delta(A)$ by
\[
\Delta(A):=\sup\set{\delta^-(A;a)\mid a\in A\setminus \{\inf A\}}.
\]
If $A$ is a closed set, we have $a-\Delta (A;a)\in A$.

\begin{lem}\label{lem:genInverse}
For any $F$ as above, we have the following \eqref{lem:genInverse:eq1}, \eqref{lem:genInverse:eq2}, and
\eqref{lem:genInverse:eq3}.
\begin{enumerate}
\item $F \circ\tilde F(s) \ge s$ for any real number $s$ with $0 \le s \le1$. \label{lem:genInverse:eq1}
\item $\tilde F\circ F(t) \le t$ for any real number $t$ with $F(t) > 0$.\label{lem:genInverse:eq2}
\item $F^{-1}((-\infty, t]) \setminus \{0\} = ( 0, F(t) ]$ for any real number $t$.\label{lem:genInverse:eq3}
\end{enumerate}
\end{lem}
The proof of the lemma is straight forward and omitted (see \cite{Nkj:max}).

\begin{proof}[Proof of Theorem \ref{thm:isoICL} \eqref{thm:isoICL:ICLtoDominant}]
Let $V$ be the cumulative distribution
function of $\nu$.
Take any 1-Lipschitz function $f:X\to \R$ and let $F:\R\to[0,1]$ be the cumulative distribution function of $f_*\mux$.
We put $\pi:=(\tilde V,\tilde F)_*\mathcal L^1|_{[0,1]}$ and see $\pi\in \Pi(\nu,f_*\mux)$.
It suffices to prove $\dev_{\succ}\supp\pi\le\varepsilon+\delta$, where $\delta:=\Delta(\supp\nu)$. Take any points $(x_i,y_i)\in\supp\pi$ for $i=1,2$. Let us prove
\begin{equation}\label{eq:ICLtoDominantGoal}
y_2-y_1-\max\set{x_2-x_1,0}\le \varepsilon+\delta.
\end{equation}
Since $\{0\}$ is a null set with respect to $\mathcal L^1$, we have
\begin{align*}
\supp\pi&=\supp(\tilde V,\tilde F)_*\Leb{[0,1]}\\
&\subset\overline{(\tilde V,\tilde F)(\supp \Leb{[0,1]}\setminus \{0\})}\\
&=\overline{(\tilde V,\tilde F)((0,1])}.
\end{align*}
Then, there exists $\{t_i^n\}_{n=1}^\infty\subset (0,1]$ such that $x_i=\lim_{n\to\infty}\tilde V(t_i^n)$ and $y_i=\lim_{n\to\infty}\tilde F(t_i^n)$ for $i=1,2$.

If we have $x_1>x_2$, we see $y_1\ge y_2$, which implies \eqref{eq:ICLtoDominantGoal}. In fact, we have
\[
y_2-y_1-\max\set{x_2-x_1,0}=y_2-y_1\le 0\le\varepsilon+\delta.
\]
We assume $x_1\le x_2$.
Let us prove
\begin{equation}\label{eq:ICLtoDominant}
V\circ\tilde V(t_2^n)\le F(\tilde F(t_1^n)+\tilde V(t_2^n)-\tilde V(t_1^n)+\delta+\varepsilon)
\end{equation}
for any positive integer $n$.
In the case that $\tilde V(t_1^n)=\inf\supp\nu$, we have
\[
0<t_1^n\le F\circ \tilde F(t_1^n)=\mux(f^{-1}((-\infty,\tilde F(t_1^n)])),
\]
which implies
\[
V\circ\tilde V(t_1^n)=\nu(\{\inf\supp\nu \})\le \mux(f^{-1}((-\infty,\tilde F(t_1^n)]))
\]
by the assumption of this theorem, By using $\ICL_\varepsilon(\nu)$, we obtain
\begin{align*}
V\circ\tilde V(t_2^n)&\le\mux(B_{\tilde V(t_2^n)-\tilde V(t_1^n)+\varepsilon}(f^{-1}((-\infty,\tilde F(t_1^n)])))\\
&\le\mux(f^{-1}(B_{\tilde V(t_2^n)-\tilde V(t_1^n)+\varepsilon}((-\infty,\tilde F(t_1^n)])))\\
&=F(\tilde F(t_1^n)+\tilde V(t_2^n)-\tilde V(t_1^n)+\varepsilon).
\end{align*}
In the case that $\tilde V(t_1^n)>\inf\supp\nu$, we have $\Delta(\supp\nu; \tilde V(t_1^n)) <\infty$. By the definition of $\Delta(\supp\nu; \tilde V(t_1^n))$, there exists a sequence $\{s_k^n \}_{k=1}^{\infty}$ of positive real numbers such that $\lim_{k\to\infty}s_k^n=\Delta(\supp\nu; \tilde V(t_1^n))$ and $\tilde V(t_1^n)-s_k^n\in\supp\nu$ for any positive integer $k$. By the definition of $\tilde V$, we have $V(\tilde V(t_1^n)-s)<t_1^n$ for any real number $s>0$, which implies
\begin{align*}
V(\tilde V(t_1^n)-s_k^n)<t_1^n
\le F\circ\tilde F(t_1^n)=\mux(f^{-1}((-\infty,\tilde F(t_1^n)])).
\end{align*}
By $\ICL_\varepsilon(\nu)$, we have
\begin{align*}
V\circ\tilde V(t_2^n)&\le \mux(B_{\tilde V(t_2^n)-\tilde V(t_1^n)+s_k^n+\varepsilon}(f^{-1}((-\infty,\tilde F(t_1^n)])))\\
&\le \mux(f^{-1}(B_{\tilde V(t_2^n)-\tilde V(t_1^n)+s_k^n+\varepsilon}((-\infty,\tilde F(t_1^n)])))\\
&=F(\tilde F(t_1^n)+\tilde V(t_2^n)-\tilde V(t_1^n)+s_k^n+\varepsilon).
\end{align*}
By taking limits with respect to $k$, we have
\begin{align*}
V\circ\tilde V(t_2^n)&\le F(\tilde F(t_1^n)+\tilde V(t_2^n)-\tilde V(t_1^n)+\Delta(\supp\nu; \tilde V(t_1^n))+\varepsilon)\\
&\le F(\tilde F(t_1^n)+\tilde V(t_2^n)-\tilde V(t_1^n)+\delta+\varepsilon).
\end{align*}
Thus we obtain \eqref{eq:ICLtoDominant}.

By using \eqref{eq:ICLtoDominant}, we have
\begin{align*}
t_2^n\le V\circ\tilde V(t_2^n)=F(\tilde F(t_1^n)+\tilde V(t_2^n)-\tilde V(t_1^n)+\delta+\varepsilon).
\end{align*}
Since $\tilde F$ is monotone non-decreasing, we have
\begin{align*}
\tilde F(t_2^n)&\le\tilde F\circ F(\tilde F(t_1^n)+\tilde V(t_2^n)-\tilde V(t_1^n)+\delta+\varepsilon)\\
&\le \tilde F(t_1^n)+\tilde V(t_2^n)-\tilde V(t_1^n)+\delta+\varepsilon.
\end{align*}
By taking limits with respect to $n$, we obtain $y_2-y_1\le x_2-x_1+\delta+\varepsilon$. This completes of proof.
\end{proof}

\begin{proof}[Proof of Theorem \ref{thm:isoICL} \eqref{thm:isoICL:DominantToICL}]
Take any two real numbers $a,b\in\supp\nu$ with $a\le b$ and any Borel set $A\subset X$ with $\mux(A)>0$ and $\mux(A)\ge V(a)$.
We define a 1-Lipschitz function $f:X\to\R$ by $f(x):=\dx(x,A)$ for $x\in X$. 
Since $\nu$ is an $\varepsilon$-iso-dominant of $X$, there exists a transport plan between between $\nu$ and $f_*\mux$ such that $\dev_{\succ}\supp\pi\le\varepsilon$.
We put
\begin{align*}
a'&:=\sup\set{x\mid (x,y)\in \supp\pi\cap (\R\times(-\infty,0])},\\
b'&:=\sup\set{x\mid (x,y)\in \supp\pi\cap (\R\times(-\infty,b-a+\varepsilon])}.
\end{align*}
We remark that we have $a'\le b'$ by the definition of $a'$ and $b'$.
Now, we have
\begin{align*}
V(a)&\le \mux(A)\le f_*\mux((-\infty,0])=\pi(\R\times(-\infty,0])\\
&=\pi((-\infty,a']\times (-\infty,0])=V(a').
\end{align*}
In particular, we have 
\begin{equation}\label{thm:isoICL:DominantToICL:compare_A}
a'\ge\inf\supp\nu
\end{equation}
because $V(a')\ge \mux(A)>0$.
Let us prove $a\le a'$. By \eqref{thm:isoICL:DominantToICL:compare_A}, we may assume $a>\inf\supp\nu$. If $a>a'$, then we have $V(a)>V(a')$ because we have $\nu(\{a \})>0$ or $\supp\nu$ is connected, which implies contradiction.

Next, let us prove $b\le b'$. We may assume $b\ge a'$ because $b\le a'\le b'$ if $b\le a'$.
By the definition of $a'$ and $\dev_\succ\pi\le\varepsilon$, there exists $y'_0\le 0$ such that $(a',y'_0)\in\supp\pi$. Similarly, there exists $y_0\in\R$ such that $(b,y_0)\in\supp\pi$ because of the definition of $b\in\supp\nu$ and $\dev_\succ\pi\le\varepsilon$.
Now, we have
\[
y_0\le y_0-y'_0 \le b-a' +\varepsilon \le b-a+\varepsilon
\]
because $\dev_\succ\pi\le\varepsilon$.
Therefore, we have $(b,y_0)\in \supp\pi\cap(\R\times (-\infty,b-a+\varepsilon])$, which implies $b\le b'$ by the definition of $b'$.

If we have
\begin{equation}\label{thm:isoICL:DominantToICL:subset}
\supp\pi\cap ((-\infty,b']\times \R)\subset (-\infty,b']\times(-\infty,b-a+2\varepsilon],
\end{equation}
then we obtain
\begin{align*}
V(b)&\le V(b')=\pi((-\infty,b']\times \R)\\
&\le \pi((-\infty,b']\times (-\infty,b-a+2\varepsilon])\\
&\le \pi(\R\times (-\infty,b-a+2\varepsilon])\\
&=f_*\mux((-\infty,b-a+2\varepsilon])\\
&=\mux(B_{b-a+2\varepsilon}(A)).
\end{align*}
Thus, the rest of the proof is to prove \eqref{thm:isoICL:DominantToICL:subset}.
Take any point $(x,y)\in\supp\pi\cap ((-\infty,b']\times \R)$. In the case that $x<b'$, there exists $(x',y')\in\supp\pi\cap(\R\times (-\infty,b-a+\varepsilon])$ such that $x'>x$ because of the definition of $b'$. Now, we have $y-y'=y-y'-\max\set{x-x',0}\le\dev_\succ\supp\pi\le\varepsilon$. Thus, we obtain $y\le y'+\varepsilon\le b-a+2\varepsilon$.

In the case that $x=b'$, for any positive integer $n$, there exists a point $(x_n,y_n)\in\supp\pi\cap(\R\times (-\infty,b-a+\varepsilon])$ such that $x-1/n<x_n\le x$.
By $\dev_{\succ}\supp\pi\le\varepsilon$, we obtain
\begin{align*}
y&\le y_n+x-x_n+\varepsilon\\
&\le x-x_n+b-a+2\varepsilon\\
&\le \frac 1n +b-a+2\varepsilon\to b-a+2\varepsilon \text{ as $n\to\infty$}.
\end{align*}
Thus we have $(x,y)\in (-\infty,b']\times(-\infty,b-a+2\varepsilon]$.
This completes the proof.
\end{proof}

\begin{prop}
Let $X$ be a finite mm-space equipped with uniform measure, and $\nu$ a Borel probability meausre on $\R$ with $N:=\#\supp\nu<\infty$. Let $\ep$ be a non-negative real number. We assume that 
\[
\Image\nu\subset (1/\# X) \mathbb Z:=\set{\frac 1{\# X}\cdot n\mid n\in\Z}.
\]
If $X$ satisfies $\IC_\varepsilon^+(\nu)$, then it satisfies $\ICL_{(N-2)\varepsilon}(\nu)$.
\end{prop}
\begin{proof}
Suppose that $X$ satisfies $\IC_\varepsilon^+(\nu)$. Take any two real number $a,b\in\supp\nu$ with $a\le b$ and a Borel subset $A\subset X$ with $\mux(A)\ge V(a)$. We may assume $a<\sup\supp\nu$.
We inductively define $\delta^+_n$ by
\[
\delta^+_1(t):=\delta^+(t)+t, \delta^+_{n+1}(t):=\delta^+\circ\delta^+_n(t)+\delta^+_n(t)
\]
for any positive integer $n$.
Now, there exists a positive integer $n_0$ such that $\delta^+_{n_0}(a)=b$ and $n_0\le N-2$.
Let us prove by induction
\begin{equation}\label{prop:icPlusToICL:eq:first}
\mux(B_{\delta^+_n(a)-a+n\varepsilon}(A))\ge V\circ\delta^+_n(a)
\end{equation}
for any positive integer $n\le n_0$.

First, we consider the case $n=1$.
Since $m_X$ is the uniform measure and $\Image\nu\subset (1/\# X) \mathbb Z$, threre exists a Borel set 
$\tilde A_1\subset A$ such that $\mux(\tilde A_1)=V(a)$ because we have $\mux(A)\ge V(a)$.
By the definition of $I_X^{\delta^+(a)+\varepsilon}$, we have
\begin{align*}
\mux(B_{\delta^+_1(a)-a+\varepsilon}(A))&=\mux(B_{\delta^+(a)+\varepsilon}(A))\\
&\ge \mux(B_{\delta^+(a)+\varepsilon}(\tilde A_1))\\
&\ge I_X^{\delta^+(a)+\varepsilon}\circ V(a)\\
&\ge V\circ \delta^+_1(a),
\end{align*}
where we remark that $X$ satisfies $\IC^+_\varepsilon(\nu)$. 

Next, we assume \eqref{prop:icPlusToICL:eq:first} for $n=k$. Thus, we have
\[
\mux(B_{\delta^+_k(a)-a+k\varepsilon}(A))\ge V\circ\delta^+_k(a),
\]
which implies that there exists a Borel subset
\[
\tilde A_k\subset B_{\delta^+_k(a)-a+k\varepsilon}(A)
\]
such that $\mux(\tilde A_k)=V\circ\delta^+_k(a)$. Therefore we have
\begin{align*}
\mux(B_{\delta^+_{k+1}(a)-a+(k+1)\varepsilon}(A))&\ge \mux(B_{\delta^+_{k+1}(a)-\delta^+_k(a)+\varepsilon}(B_{\delta^+_k(a)-a+k\varepsilon}(A)))\\
&\ge\mux(B_{\delta^+\circ \delta^+_k(a)+\varepsilon}(\tilde A_k))\\
&\ge I_X^{\delta^+\circ \delta^+_k(a)+\varepsilon}\circ V\circ \delta^+_k(a)\\
&\ge V\circ \delta^+_k(a)
\end{align*}
if $k+1\le n_0$. Thus we obtain \eqref{prop:icPlusToICL:eq:first}.
In particular, we have
\[
\mux(B_{\delta^+_{n_0}(a)-a+n_0\varepsilon}(A))\ge V\circ\delta^+_{n_0}(a).
\]
Therefore we obtain
\begin{align*}
\mux(B_{b-a+(N-2)\varepsilon}(A))&\ge \mux(B_{\delta^+_{n_0}(a)-a+n_0\varepsilon}(A))\\
&\ge V\circ\delta^+_{n_0}(a)=V(b).
\end{align*}
This completes the proof.
\end{proof}
\begin{prop}
Let $X$ be an mm-space and $\nu$ a Borel probability measure on $\R$, and $\ep\ge 0$ a real number. If $X$ satisfies $\ICL_\ep(\nu)$, then it satisfies $\IC_\ep^+(\nu)$.
\end{prop}
\begin{proof}
This follows from the definition of $\ICL_\ep(\nu)$ and $\IC_\ep^+(\nu)$.
\end{proof}

\begin{ex}\label{ex:lOneDiscreteIso}
Let $G_1,G_2,\dots,G_n$ 
be connected graphs with same order $k\ge 2$. Let $\Pi_{i=1}^n G_i$ be the cartesian product graph equipped with the path metric and the uniform measure. Let $d_0:[k]^n\to \R$ be the distance function from the origin. Then $\Pi_{i=1}^n G_i$ satisfies $\ICL((d_0)_*m_{[k]^n})$ by Theorem $13$ in \cite{Bol:comp}. Thus the measure $(d_0)_*m_{[k]^n}$ is a $1$-iso-dominant of  $\Pi_{i=1}^n G_i$ because of Theorem $\ref{thm:isoICL}$ \eqref{thm:isoICL:ICLtoDominant}.
In particular, the measure $(d_0)_*m_{[k]^n}$ is a $1$-iso-dominant of the discrete $l^1$-cube $[k]^n$.
\end{ex}

\begin{ex}\label{ex:lOneDiscreteTorusIso}
We assume that $k$ is a positive even integer. Let $X:=(\Z/(k\Z))^n$ be the discrete torus equipped with the $l^1$-distance and the uniform measure $m_X$, and $d_0:X\to \R$ the distance function from the origin.
Then it satisfies $\ICL((d_0)_*m_X)$ by Corollary 6 in \cite{Bol:isop_torus}. Thus the measure $(d_0)_*m_X$ is a $1$-iso-dominant of $X$.
\end{ex}

\section{Stability of $\varepsilon$-iso-dominant}
\begin{dfn}[$(s,t)$-iso-dominant]
Let $s$ and $t$ be two non-negative real numbers.
We call a Borel probability measure $\nu$ on $\R$ an {\it $(s,t)$-iso-dominant} of an mm-space $X$ if we have $\nu\succ'_{(s,t)}\mu$ for all $\mu\in\cM(X;1)$.
\end{dfn}
\begin{dfn}[distortion from the diagonal]
Let $(X,d_X)$ be a metric space. We define {\it the distortion from the diagonal} of a subset $S\subset X$ by
\[
\dis_\Delta S:=\sup\set{d_X(x,y) \mid (x,y)\in S}.
\]
Let $\mu$ and $\nu$ be two Borel probability meausres on $X$.
We define {\it the distortion from the diagonal} of a transport plan $\pi\in\Pi(\mu,\nu)$ between $\mu$ and $\nu$ by
\[
\dis_\Delta \pi:=\inf_S \max\set{\dis_\Delta S,1-\pi(S)}
\]
whrere $S\subset \R^2$ is a closed subset.
\end{dfn}

\begin{thm}[Strassen's theorem; cf. {\cite[Corollary 1.28]{Vil:topics}}]\label{thm:StrassenDis}
Let $\mu$ and $\nu$ be two Borel probability measures on a metric space $X$. Then we have
\[
\dP(\mu,\nu)=\inf_{\pi\in\Pi(\mu,\nu)}\dis_\Delta \pi.
\]
\end{thm}
\begin{lem}\label{lem:devIneq}
For a subset $S\subset \R^2$, we have
\[
\dev_\succ S\le 2\dis_\Delta S.
\]
\end{lem}

\begin{proof}
Take any two points $(x,y),(x',y')\in S$.
If $x-x'\ge 0$, then we have
\begin{align*}
y-y'-\max\set{x-x',0}&=y-y'-|x-x'|\\
&\le |y-y'|-|x-x'|\\
&\le |x-y|+|x'-y'|\le 2\dis_\Delta S.
\end{align*}
If $x-x'<0$, then we have
\begin{align*}
y-y'-\max\set{x-x',0}&=y-y'\\
&<y-y'+x'-x\\
&\le |y-x|+|x'-y'|\le 2\dis_\Delta S.
\end{align*}
Thus we obtain $\dev_\succ S\le 2\dis_\Delta S$. This completes the proof.
\end{proof}

\begin{lem}\label{lem:order_dev:Prohorov}
Let $\mu$ and $\nu$ be two Borel probability measures on $\R$. If $d_P(\mu,\nu)<\varepsilon$, then we have $\mu\succ'_{(2\varepsilon,\varepsilon)}\nu$.
\end{lem}
\begin{proof}
This follows from Theorem \ref{thm:StrassenDis} and Lemma \ref{lem:devIneq}.
\end{proof}

\begin{lem}\label{lem:isodominantProk}
Let $\mu$ and $\nu$ be two Borel probability meaures on $\R$, and $X$ an mm-space. If  $\mu$ is an $(s,t)$-iso-dominant of $X$ and we have $d_P(\mu,\nu)<\varepsilon$, then $\nu$ is an $(s+2\varepsilon,t+\varepsilon)$-iso-dominant of $X$.
\end{lem}

\begin{proof}
This follows from Lemma \ref{lem:order_dev:Prohorov} and Theorem \ref{thm:order_dev:transitive}.
\end{proof}

\begin{lem}\label{lem:isodominantConc}
Let $X$ and $Y$ be two mm-spaces, and $\nu$ a Borel probability measure on $\R$.
If $\nu$ is an $(s,t)$-iso-dominant of $X$ and we have $d_{conc}(X,Y)<\varepsilon$, then $\nu$ is an $(s+2\varepsilon,t+\varepsilon)$-iso-dominant of $Y$
\end{lem}

\begin{proof}
Take any $g\in\Lip_1(Y)$. By $\dconc(X,Y)<\varepsilon$, there exists two parameters $\varphi:I\to X$ and $\psi:I\to Y$ such that 
\[
d_{\rm H}^{\rm KF}(\varphi^*\Lip_1(X),\psi^*\Lip_1(Y))<\ep.
\]
Thus there exists $f\in\Lip_1(X)$ such that $\kf(\varphi^*f,\psi^*g)<\varepsilon$.
Now, we have
\begin{align*}
\dP(f_*m_X,g_*m_Y)&=\dP(f_*(\varphi_*\leb),g_*(\psi_*\leb))\\
&\le \kf(\varphi^*f,\psi^*g)<\varepsilon.
\end{align*}
Therefore we have $f_*m_X\succ_{(2\varepsilon,\varepsilon)}g_*m_Y$ by Lemma \ref{lem:order_dev:Prohorov}.
Since $\nu$ is an $(s,t)$-iso-dominant of $X$, we have  $\nu\succ_{(s,t)} f_*m_X$, which implies $\nu \succ_{(s+2\varepsilon,t+\varepsilon)}g_*m_Y$ by Theorem \ref{thm:order_dev:transitive}.
\end{proof}

\begin{proof}[Proof of Theorem \ref{thm:stabilityIso}]
Without loss of generality, we assume
\[
\dconc(X_n,X)<\varepsilon_n \text{ and } \dP(\nu_n,\nu)<\varepsilon_n \text{\quad for any positive integer $n$}.
\]
Take any positive integer $n$.
Since the measure $\nu_n$ is an $(s+\varepsilon_n,t+\varepsilon_n)$-iso-dominant of $X_n$, the measure $\nu$ is an $(s+3\varepsilon_n,t+2\varepsilon_n)$-iso-dominant of $X_n$ by Lemma \ref{lem:isodominantProk}. By Lemma \ref{lem:isodominantConc}, the meaure $\nu$ is an $(s+5\varepsilon_n,t+3\varepsilon_n)$-iso-dominant of $X$.  Thus we have $\nu\succ_{(s+5\varepsilon_n,t+3\varepsilon_n)} f_*m_X$ for any $f\in\Lip_1(X)$.
By Theorem \ref{thm:order_dev:nonDegenerate}, we obtain $\nu\succ_{(s,t)}f_*m_X$. This completes the proof.
\end{proof}

To apply Theorem \ref{thm:stabilityIso} for pyramids, we consider the following Propositions \ref{prop:orderIsoDominant} and \ref{prop:equivIsoDominant}, and Definition \ref{dfn:isoDominantPyramid}. 
We refer to \cite{Gmv:green, Shioya:mmg} for the theory of pyramids.

\begin{prop}\label{prop:orderIsoDominant}
Let $X$ and $Y$ be two mm-spaces. 
If a Borel probability measure $\nu$ on $\R$ is an $(s,t)$-iso-dominant of $X$ for $s,t\ge 0$ and we have $X\succ Y$, then $\nu$ is an $(s,t)$-iso-dominant of $Y$.
\end{prop}

\begin{dfn}\label{dfn:isoDominantPyramid}
Let $\cY\subset\cX$.
We say that a Borel probability measure $\nu$ on $\R$ is {\it an $(s,t)$-iso-dominant of $\cY$} if $\nu$ is an $(s,t)$-iso-dominant of $X$ for any mm-space $X$.
\end{dfn}

\begin{prop}\label{prop:equivIsoDominant}
Let $X$ be an mm-space, and $\nu$ a Borel probability measure on $\R$.
Then, $\nu$ is an $(s,t)$-iso-dominant of $X$ if and only if $\nu$ is an $(s,t)$-iso-dominant of $\mathcal P_X:=\set{Y\in \cX\mid Y\prec X}$.
\end{prop}

\begin{thm}\label{thm:isoStabilityPyramid}
Let $\cY_n\subset \cX$ be a $\Box$-closed subset, and $\overline{\cY}_\infty$ the set of the limits of convergent subsequences of $Y_n\in\cY_n$. We assume that a sequence $\{\nu_n\}_{n=1}^\infty$ of Borel probability measures converges weakly to a Borel probability measure $\nu$, and a sequence $\{\ep_n\}_{n=1}^\infty$ of non-negative real numbers converges to $0$. If $\nu_n$ is an $(s+\varepsilon_n,t+\varepsilon_n)$-iso-dominant of $\cY_n$ for any positive integer $n$, then $\nu$ is an $(s,t)$-iso-dominant of $\overline{\cY}_\infty$.
\end{thm}

\begin{proof}
This theorem follows by Theorem \ref{thm:stabilityIso} and Proposition \ref{prop:concBox}.
\end{proof}

We obtain the following corollary by Proposition 6.9 in \cite{Shioya:mmg}.

\begin{cor}
Let $\{\mathcal P_n\}_{n=1}^\infty$ be a sequence of pyramids, and $\{\nu_n\}_{n=1}^\infty$ a sequence of Borel probability measures on $\R$. We assume that $\{\mathcal P_n\}_{n=1}^\infty$ converges weakly to a pyramid $\mathcal P$  and $\{\nu_n\}_{n=1}^\infty$ converges weakly to a Borel probability measure $\nu$ on $\R$. 
If $\nu_n$ is an $(s+\varepsilon_n,t+\varepsilon_n)$-iso-dominant of $\mathcal P_n$, then $\nu$ is an $(s,t)$-iso-dominant of $\mathcal P$.
\end{cor}

\begin{proof}[Proof of Theorem $\ref{thm:NormalLevyProductGraphs}$]
We define a function $d_0:\R^n\to\R$ by 
\[
d_0((x_i)_{i=1}^n):=\sum_{i=1}^n |x_i|.
\]
By Example \ref{ex:lOneDiscreteIso}, the measure $(d_0)_*m_{[k]^n}$ is a $1$-iso-dominant of $X_n$, which implies that $(\ep_n \cdot d_0)_*m_{[k]^n}$ is an $\ep_n$-iso-dominant of $Y_n$ by Proposition \ref{prop:isoScaling}. By the central limit theorem, $(\ep_n \cdot d_0)_*m_{[k]^n}$ converges weakly to $\gamma^1$ as $n\to\infty$. We put $\cY_n:=\{Y_n \}$ and $\nu$ is an iso-dominant of $\overline{\cY}_\infty$ by Theorem \ref{thm:isoStabilityPyramid}. This completes the proof.
\end{proof}
\section{Applications of Iso-Lipschitz order with an additive error}
\subsection{Isoperimetric inequality of non-discrete $l^1$-cubes}

In this section, we assume that $[0,1]^n$ is equipped with the $l^1$-distance $d_{l^1}$ and the uniform measure $m_{[0,1]^n}:=\mathcal L^n|_{[0,1]^n}$, where $\mathcal L^n$ is the $n$-dimensional Lebesgue measure.
Put $[k]:=\{0,1,2,\dots ,k-1 \}$.
We have $\frac 1k [k]=\{0,\frac 1k,\frac 2k,\dots ,1-\frac 1k \}\subset [0,1]$.
We assume that $\frac 1k [k]^n$ is equipped with $l^1$-distance $d_{l^1}$ and the uniform measure $m_{\frac 1k[k]^n}:=\frac 1{k^n}\sum_{x\in \frac 1k[k]^n} \delta_x$.

\begin{lem}\label{lem:discretization}
The sequence $\{ m_{\frac 1k [k]}\}_{k=1}^\infty$ converges weakly to $m_{[0,1]^n}$ as $k\to\infty$.
\end{lem}
\begin{proof}
Define a function $f:[0,1]^n\to \frac 1k [k]^n$ by
$f((x_i)_{i=1}^n):=(\frac 1k \lfloor k x_i \rfloor)_{i=1}^n$, where $\lfloor \cdot \rfloor$ is the floor function.
Then we have $(\id_{[0,1]^n},f)_*m_{[0,1]^n}\in\Pi(m_{[0,1]^n},m_{\frac 1k [k]})$.
Take any point $(x,f(x))\in \supp\pi=(\id_{[0,1]^n},f)([0,1]^n)$ and put $x:=(x_i)_{i=1}^n$.
Since
\[
d_{l^1}(x,f(x))=\sum_{i=1}^n |x_i-\frac 1k \lfloor kx_i\rfloor|\le \frac nk,
\]
we have $\dis_\Delta \supp\pi\le\frac nk$.
By Theorem \ref{thm:StrassenDis}, we obtain
\[
\dP(m_{[0,1]^n},m_{\frac 1k [k]})\le \frac nk \to 0
\]
as $k\to\infty$. This completes the proof.
\end{proof}

\begin{proof}[Proof of Theorem $\ref{thm:lOneIso}$]
We define a function $d_0:\R^n\to\R$ by $d_0((x_i)_{i=1}^n):=\sum_{i=1}^n |x_i|$. By Example \ref{ex:lOneDiscreteIso}, the measure $(d_0)_*m_{[k]^n}$ is a $1$-iso-dominant of $[k]^n$.
Thus the measure $(\frac 1k d_0)_*m_{[k]^n}$ is a $\frac 1k$-iso-dominant of $\frac 1k[k]^n$ because of Proposition \ref{prop:isoScaling}.
Since $d_0$ is 1-Lipschitz,
we have
\begin{align*}
\dP((\frac 1k d_0)_*m_{[k]^n},(d_0)_*m_{[0,1]^n})&= \dP((d_0)_*m_{\frac 1k[k]^n}, (d_0)_*m_{[0,1]^n})\\
&\le \dP(m_{\frac 1k[k]^n},m_{[0,1]^n})\le \frac nk
\end{align*}
by Lemma \ref{lem:discretization}.
By Theorem \ref{thm:stabilityIso}, the measure $(d_0)_*m_{[0,1]^n}$ is an iso-dominant of $[0,1]^n$.
This completes the proof.
\end{proof}

We obtain Theorem \ref{thm:lOneTorusIso} in the same way as in the proof of Theorem \ref{thm:lOneIso} by using Example \ref{ex:lOneDiscreteTorusIso}.

\subsection{Comparison theorem for observable diameter}

\begin{prop}\label{prop:orderDiam}
Let $\mu$ and $\nu$ be two Borel probability measures on $\R$. If $\mu\succ'_{(s,t)}\nu$, then we have
\[
\diam(\mu;1-\kappa)+s \ge \diam (\nu;1-\kappa-t) \text{\quad for any $\kappa> 0$.}
\]
\end{prop}
\begin{proof}
By $\mu\succ'_{(s,t)}\nu$, there exist $\pi\in\Pi(\mu,\nu)$ and a Borel set $S\subset \R^2$ such that $\dev_\prec S\le s$ and $1-\pi(S)\le t$.
Take any Borel set $A\subset \R$ with $\mu(A)\ge 1-\kappa$.
Put $B:=\pr_2(S\cup (\pr_1)^{-1}(A))$.
Since
\begin{align*}
\nu(\overline B)&\ge\pi(S\cup (\pr_1)^{-1}(A))\\
&=\pi( (\pr_1)^{-1}(A))-\pi(S^c\cup (\pr_1)^{-1}(A))\\
&\ge\mu(A)-\pi(S^c)\ge1-\kappa-t,
\end{align*}
we have $\diam(\nu;1-\kappa-t)\le\diam B$.
By Lemma \ref{lem:dev_minus}, we have $\diam B\le \diam A+\dev_\prec S$, which implies $\diam(\nu;1-\kappa-t)\le \diam A+s$. Then we obtain $\diam(\nu;1-\kappa-t)\le\diam(\mu;1-\kappa)+s$. This completes the proof.
\end{proof}

\begin{prop}\label{prop:diamWeakly}
Let $\{\mu_n\}_{n=1}^\infty$ be a sequence of Borel probability measures on $\R$ and $\kappa$ a positive real number. We assume that $\{\mu_n\}_{n=1}^\infty$ converges weakly to a Borel probablity meaure $\mu$ on $\R$ and that the function $t\mapsto \diam(\mu;1-t)$ is continuous at $\kappa$. Then we have 
\[
\lim_{n\to\infty}\diam(\mu_n;1-\kappa)=\diam(\mu;1-\kappa).
\]
\end{prop}
\begin{proof}
Put $\ep_n:=\dP(\mu_n,\mu)+\frac 1n$. By Lemma \ref{lem:order_dev:Prohorov}, we have $\mu_n \prec_{(2\ep_n,\ep_n)} \mu$ and $\mu \prec_{(2\ep_n,\ep_n)} \mu_n$. Since $\kappa-\ep_n>0$ for sufficiently large $n$, we have
\begin{align*}
\diam(\mu;1-(\kappa-\ep_n))+2\ep_n&\ge \diam(\mu_n;1-(\kappa-\ep_n)-\ep_n)\\
&\ge \diam(\mu;1-\kappa-\ep_n)-2\ep_n
\end{align*}
by Proposition \ref{prop:orderDiam}. Since $t\mapsto \diam(\mu;1-t)$ is continuous, we obtain
$\lim_{n\to\infty}\diam(\mu_n;1-\kappa)=\diam(\mu;1-\kappa$. This completes the proof.
\end{proof}
\begin{thm}\label{thm:isoDominantToObsDiam}
Let $s$ and $t$ be two non-negative real numbers.
If a Borel probability measure $\nu$ on $\R$ is an $(s,t)$-iso-dominant of an mm-space $X$, then we have
$\ObsDiam(X;-\kappa-t)\le \diam(\nu;1-\kappa)+s$ for any $\kappa\ge 0$.
\end{thm}

\begin{proof}
Take any 1-Lipschitz function $f:X\to\R$. Since $\nu$ is an $(s,t)$-iso-dominant of $X$, we have $\nu\prec_{(s,t)} f_*m_X$. By Proposition \ref{prop:orderDiam}, we have $\diam(f_*m_X;1-\kappa-t)\le\diam(\nu;1-\kappa)+s$. Thus we obtain $\ObsDiam(X;1-\kappa-t)\le \diam(\nu;1-\kappa)+s$. This completes the proof.
\end{proof}
Let $G_1,G_2,\dots,G_n,\dots$
be connected graphs with same order $k\ge 2$.
Put $\varepsilon_{k,n}:=\sqrt{\frac{12}{(k^2-1)n}}$.
\begin{thm}\label{thm:graphObsDiam}
We define a function $d_{0,n}:\R^n\to \R$ by $d_{0,n}((x_i)_{i=1}^n):=\sum_{i=1}^n |x_i|$.
Put $\nu_{k,n}:=(\ep_{k,n}\cdot d_{0,n})_*m_{[k]^n}$.
Then we have
\begin{align}
\ObsDiam(\varepsilon_{k,n}\prod_{i=1}^n G_i;-\kappa)&\le \diam(\nu_{k,n};1-\kappa)+\varepsilon_{k,n}\label{thm:graphObsDiam:ineq1}\\
&\le \ObsDiam(\varepsilon_{k,n}[k]^n;-\kappa)+\varepsilon_{k,n}\label{thm:graphObsDiam:ineq2}\\
&\le \diam(\nu_{k,n};1-\kappa)+2\varepsilon_{k,n}.\label{thm:graphObsDiam:ineq3}
\end{align}
\end{thm}

\begin{proof}
By Theorem \ref{thm:isoDominantToObsDiam} and Example \ref{ex:lOneDiscreteIso}, and Proposition \ref{prop:isoScaling}, we have \eqref{thm:graphObsDiam:ineq1} and \eqref{thm:graphObsDiam:ineq3}. Since $\nu_{k,n}\in \cM(\varepsilon_{k,n}[k]^n;1)$, we have \eqref{thm:graphObsDiam:ineq2}. This completes the proof.
\end{proof}
\begin{cor}
We have
\[
\limsup_{n\to\infty}\ObsDiam(\varepsilon_{k,n}\prod_{i=1}^n G_i;-\kappa)\le\diam(\gamma^1;1-\kappa) \text{ for $\kappa> 0$}.
\]
\end{cor}
\begin{proof}
This follows from \eqref{thm:graphObsDiam:ineq1} in Theorem \ref{thm:graphObsDiam} and Proposition \ref{prop:diamWeakly}. This completes the proof.
\end{proof}
\begin{cor}
We have
\[
\lim_{n\to\infty}\ObsDiam(\varepsilon_{k,n}[k]^n;-\kappa)=\diam(\gamma^1;1-\kappa) \text{ for $\kappa> 0$}.
\]
In paricular, we obtain
\[
\lim_{n\to\infty}\ObsDiam(\frac{2}{\sqrt{n}}Q^n;-\kappa)=\diam(\gamma^1;1-\kappa) \text{ for $\kappa> 0$}
\]
as the case $k=2$.
\end{cor}
\begin{proof}
This follows from \eqref{thm:graphObsDiam:ineq2} and \eqref{thm:graphObsDiam:ineq3} in Theorem \ref{thm:graphObsDiam} and Proposition \ref{prop:diamWeakly}. This completes the proof.
\end{proof}

\section*{Acknowledgment}

The author would like to thank Prof. Takashi Shioya for many helpful suggestions. He also thanks Daisuke Kazukawa for many stimulating discussions.

  
\end{document}